\documentclass{amsart}
\usepackage{amsfonts}
\usepackage{amsmath,amssymb}	
\usepackage{amsthm}
\usepackage{amscd}
\usepackage{graphics}
\usepackage{graphicx}

{
\theoremstyle{definition}
\newtheorem{Def}{Definition}
\newtheorem{Ex}{Example}
\newtheorem{Rem}{Remark}
\newtheorem{Prob}{Problem}
}

\newtheorem{Cor}{Corollary}
\newtheorem{Prop}{Proposition}
\newtheorem{Thm}{Theorem}

\newtheorem{Lem}{Lemma}

\begin{document}
\title[Lifts of spherical Morse functions]{Lifts of spherical Morse functions}

\author{Naoki Kitazawa}
\address{Institute of Mathematics for Industry, Kyushu University, 744 Motooka, Nishi-ku Fukuoka 819-0395, Japan}
\email{n-kitazawa.imi@kyushu-u.ac.jp}
\maketitle
\begin{abstract}
In studies of smooth maps with good differential topological properties such as immersions and embeddings, whose codimensions are positive, and Morse functions and their higher dimensional versions including {\it fold} maps and so-called {\it stable} maps whose codimensions are 0 or minus, and application to algebraic and differential topology of manifolds, or the theory of global singularity, liftings or disingularizations
 of maps in specific classes are fundamental and important problems. A {\it lift} of a smooth map into a Euclidean space is a smooth map into a higher dimensional Euclidean space such that the map obtained as the composition with the canonical projection is the original map. 

In this paper, as a new problem of constructing lifts, we consider Morse functions such that inverse images of regular values are disjoint unions of spheres or {\it spherical} Morse functions defined and systematically studied by Saeki and Suzuoka in 2000s, which are regarded as generalizations of Morse functions with just two singular
 points on homotopy spheres, and lift them to immersions, embeddings and stable {\it special generic maps}, regarded as higher dimensional versions of Morse functions with just two singular points on homotopy spheres. Note that constructing lifts of maps to special generic maps or more general generic smooth maps which may have singular points, is a new work.   
 
\end{abstract}

\section{Introduction.}
\label{sec:1}
Smooth maps with good differential topological properties such as immersions and embeddings and Morse functions and their higher dimensional versions including 
{\it fold} maps and so-called {\it stable} maps are important objects and tools in the studies of the global singularity theory, or studies of global differential topological properties of smooth manifolds.

Let $a>\geq b$ be positive integers and denote the canonical projection

$$(x_1, \cdots,x_b,\cdots,x_a) \mapsto (x_1, \cdots,x_b)$$

by ${\pi}_{a,b}:{\mathbb{R}}^a \rightarrow {\mathbb{R}}^b$. 
Let $m$ and $n$ be positive integers, $M$ be a smooth closed manifold of dimension $m$ and let $f:M \rightarrow {\mathbb{R}}^n$ be a smooth map.  In this paper, we consider
 the following fundamental and important problem.

\begin{Prob}
For a positive integer $k>0$, does there exist a smooth map $f_k$ in a fixed class satisfying $f={\pi}_{n+k,n} \circ f_k$ ?
\end{Prob}

This is a so-called lifting problem of a smooth map and we call $f_k$ a {\it lift} of $f$. We also say that $f$ is {\it lifted} to $f_k$ and that we lift $f$ to $f_k$. There have been shown various explicit related results as the following since 1950--60s. 

\begin{itemize}
\item Plane curves can be lifted to classical knots in the $3$-dimensional space.
\item Morse functions and stable
 maps of $2$ or $3$-dimensional manifolds into the plane are lifted to immersions or embeddings into appropriate dimensional spaces. See \cite{haefliger}, \cite{levine}, \cite{yamamoto2} and \cite{yamamoto3} for example.
\item Generic maps between equi-dimensional manifolds are lifted to immersions or embeddings into one-dimensional higher spaces under suitable conditions. See \cite{saito} and see also \cite{blankcurley} for example.
\item Various {\it special generic maps}, which are regarded as higher dimensional versions of Morse functions with just two singular points and reviewed in section \ref{sec:2}, are lifted to immersions or embeddings of appropriate dimensions (\cite{saekitakase} and see also \cite{nishioka2}).
\end{itemize}
 
Note that in these studies, maps are lifted to immersions or embeddings and in the present paper, we lift the maps
 to maps whose codimensions are not always positive. More explicitly, we lift Morse functions such that inverse images of regular
 values are disjoint unions of spheres, which are regarded as generalizations of special generic functions from the viewpoint of the fact that manifolds appearing as inverse images of regular values are disjoint unions of spheres, to not only immersions and embeddings but also special generic maps, 

We present the contents of this paper. In the next section, we explain stable maps slightly and review {\it fold} maps and {\it special generic} maps. As main objects of the paper, we also
 define classes of fold maps and stable maps such that the inverse image
 of each regular value is a disjoint union of spheres, which are regarded as generalizations of special generic maps and have been systematically studied
 in \cite{saekisuzuoka} for example. 
Next, we show several theorems on lifts of Morse functions belonging to the class defined before, starting from preliminaries for the proofs. 
Most of theorems are on constructing lifts to immersions, embeddings and special generic maps. As fundamental and new-type studies, we also study about classes of maps obtained
 as the compositions of obtained lifts and canonical projections to spaces whose dimensions are larger than those of the target spaces of the original maps. For example, we study about if a composition is special generic or not.

Throughout this paper, maps, maps between manifolds including homotopies and isotopies, bundles whose bases spaces and fibers are manifolds are smooth and of class $C^{\infty}$ unless otherwise stated. However, to emphasize that we are discussing in the smooth category, we denote "smoothly isotopic" for exampke. We call the set of all singular points of a smooth map the {\it singular set}, the image of the singular set the {\it singular value set}, the complement of the singular value set the {\it regular value set}.

As an important terminology, a map between topological spaces is said to be {\it proper} if the inverse image of any compact subspace in the target space by the map is compact. 

\section{Morse functions, fold maps, special generic maps and stable maps.}
\label{sec:2}
{\it Stable} maps are essential smooth maps in higher dimensional versions
 of the theory of Morse functions. We introduce the definition and for this, we need to introduce that two smooth maps $c_1:X_1 \rightarrow Y_1$ and $c_2:X_2 \rightarrow Y_2$ are said to be {\it $C^{\infty}$ equivalent} if there exists a pair $({\phi}_X:X_1 \rightarrow X_2,{\phi}_Y:Y_1 \rightarrow Y_2)$ of diffeomorphisms satisfying the relation ${\phi}_Y \circ c_1=c_2 \circ {\phi}_X$. Two smooth maps are said to be {\it $C^{0}$ equivalent} if in the relation, two diffeomorphisms are homeomorphisms.

A {\it stable} map is a smooth map such that maps in an open neighborhood of the map in the
 Whitney $C^{\infty}$ topology and the original map are $C^{\infty}$ equivalent.
As simplest examples, Morse functions, whose singular sets are discrete, and their higher dimensional versions, {\it fold maps}, are stable if and only if their restrictions to the singular sets, which are smooth closed submanifolds and the restrictions to which are smooth
 immersions, are transversal in the case where we only consider proper maps: note that at distinct singular points of stable Morse functions, the values are distinct. Stable Morse functions exist densely and stable maps exist densely on a closed
 manifold if the pair of dimensions of the source and the target manifolds is {\it nice}. Studying about existence of (stable) fold maps is a fundamental, difficult and interesting problem (\cite{eliashberg}, \cite{eliashberg2} etc.).
 See also \cite{golubitskyguillemin} for fundamental facts on stable maps and such theory.  

We introduce the definition and a fundamental property of a {\it fold} map (see also \cite{saeki} for example for introductory facts and advanced studies on fold maps). 

\begin{Def}
\label{def:1}
A smooth map is said to be a {\it fold} map if at each singular point $p$ it is
 of the form $$(x_1,\cdots,x_m) \mapsto (x_1,\cdots,x_{n-1},\sum_{k=n}^{m-i(p)}{x_k}^2-\sum_{k=m-i(p)+1}^{m}{x_k}^2)$$ for some
     integers $m,n,i(p)$ satisfying $m \geq n \geq 1$ and $0 \leq i(p) \leq \frac{m-n+1}{2}$.
 A singular point of a smooth map is said to be a {\it fold} point if at the point the function is such a form. 
\end{Def}

\begin{Prop}
\label{prop:1}
In Definition \ref{def:1}, the integer $i(p)$ is taken as a non-negative integer not larger than $\frac{m-n+1}{2}$ uniquely.
 We call $i(p)$ the {\it index} of $p$. The
 set of all singular points of an index of a fold map is a smooth closed submanifold of dimension $n-1$ and the
     restriction of a fold map to the singular set is an immersion as introduced before.  
\end{Prop}

\begin{Def}
\label{def:2}
A fold map such that the indices of all singular points are $0$ is said to be {\it special generic}.
\end{Def}

Note that a fold map is at a singular point, locally regarded as a product of a Morse function with just one singular point in the interior and the identity map on an open ball. A special generic map is at a singular point locally regarded as a product of a height function of a standard closed disc or a unit disc and the identity map on an open ball. 

\begin{Def}
\label{def:3}
A proper stable (fold) map from a closed or open manifold of dimension $m$ into ${\mathbb{R}}^n$ satisfying $m \geq 1$ is said to be {\it spherical} ({\it standard-spherical}) if the inverse image
 of each regular value is a disjoint union of (resp. standard) spheres (or points).
For a spherical fold map, it is said to be {\it normal} if the connected component containing a singular point of the inverse image of a
 small interval intersecting with the singular value set at once and in its interior is either of the following.
\begin{enumerate}
\item The ($m-n+1$)-dimensional standard closed disc.
\item A manifold PL homeomorphic to an ($m-n+1$)-dimensional compact manifold obtained by removing the interior of three disjoint ($m-n+1$)-dimensional smoothly embedded closed discs from the ($m-n+1$)-dimensional standard sphere.
\end{enumerate}  
\end{Def}

It easily follows that the indices of singular points are $0$ or $1$.

\begin{Ex}
\label{ex:1}
A Morse function with just two singular points on a closed manifold, characterizing homotopy spheres topologically (except $4$-dimensional homotopy spheres being not the standard $4$-sphere, or exotic $4$-dimensional spheres, which are not discovered), is a special generic map.
It is also a normal standard-spherical fold map.
A canonical projection of a unit sphere is also a stable special generic map. It is also a normal standard-spherical fold map.
\end{Ex}

For systematic studies of spherical maps, see \cite{saekisuzuoka} for example.
In this paper, about spherical maps, we only study fold maps and do not study about spherical fold maps which are not normal.

\section{Main results and essential tools and facts for the proofs.}
\label{sec:3}
\subsection{Reeb spaces.}
\begin{Def}
\label{def:4}
Let $X$ and $Y$ be topological spaces. For points $p_1, p_2 \in X$ and for a continuous map $c:X \rightarrow Y$, we define as $p_1 {\sim}_c p_2$ if and only if $p_1$ and $p_2$ are in
 a connected component of $c^{-1}(p)$ for some $p \in Y$. The relation is an equivalence relation.
We call the quotient space $W_c:=X/{\sim}_c$ the {\it Reeb space} of $c$.
\end{Def}
\begin{Ex}
\label{ex:2}
\begin{enumerate}
\item
\label{ex:2.1}
 For a stable Morse function, the Reeb space is a graph.
\item
\label{ex:2.2}
 For a proper special generic map, the Reeb space is regarded as an immersed compact manifold with a non-empty boundary whose dimension is same as that of the target manifold. 

For a pair $(m.n)$ of integers satisfying the relation $m>n$ and a closed manifold $M$ of dimension $m$, $M$ admits a special generic map into ${\mathbb{R}}^n$ if and only if the following hold.

\begin{itemize} 
\item There exists a compact smooth manifold $P$ s.t. $\partial P \neq \emptyset$ which we can immerse into ${\mathbb{R}}^n$.
\item $M$ is obtained by gluing the following two manifolds by a bundle isomorphism between the $S^{m-n}$-bundles over the boundary $\partial P$.
\begin{itemize}
\item The total space of an $S^{m-n}$-bundle over $P$.
\item The total space of a linear $D^{m-n+1}$-bundle over $\partial P$.
\end{itemize} 
\end{itemize}

See FIGURE \ref{fig:0-1}. $P$ is regarded as the Reeb space $W_f$ of a special generic map $f$ and $\partial P$ is regarded as the singular value set. The  linear $D^{m-n+1}$-bundle over $\partial P$ is regarded as a normal bundle of the singular set, discussed in Theorem \ref{thm:8} and later.
\begin{figure}
\includegraphics[width=40mm]{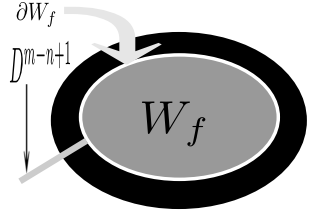}
\caption{The image of a special generic map.}
\label{fig:0-1}
\end{figure}

The Reeb space of a special generic map on a closed manifold is contractible
 if and only if the source manifold is a homotopy sphere. It is also a fundamental and important property that the quotient map from the source manifold onto the
 Reeb space induces isomorphisms on homology groups whose degrees are not larger than the difference of the dimension of the source manifold and that of the target one. For precise studies on special generic maps on closed manifolds, see \cite{saeki2} for example.  
\item
\label{ex:2.3}
 The Reeb spaces of proper stable (fold) maps are polyhedra whose dimensions and those of the target manifolds coincide. For a normal spherical fold map on a closed manifold, it is also a fundamental and important property that the quotient map from the source manifold onto the
 Reeb space induces isomorphisms on homology groups whose degrees are smaller than the difference of the dimension of the source manifold and that of the target one. See \cite{saekisuzuoka} and see also \cite{kitazawa} and \cite{kitazawa2} for example. Last, we present a height function on $S^m$ (or more generally a Morse function with just two singular points) and a height function on $S^1 \times S^{m-1}$ ($m \geq 2$) as simplest examples of normal spherical Morse functions and the Reeb spaces in FIGURE \ref{fig:0-2}.

\begin{figure}
\includegraphics[width=40mm]{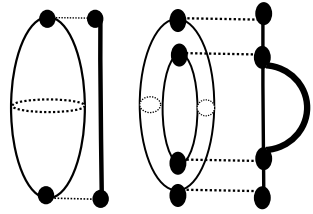}
\caption{A height function on $S^m$ (or more generally a Morse function with just two singular points) and a height function on $S^1 \times S^{m-1}$ ($m \geq 2$) as simplest examples of normal spherical Morse functions and the Reeb spaces.}
\label{fig:0-2}
\end{figure}

\item
\label{ex:2.4}
 For a proper stable map or more generally a proper {\it Thom}  map, the Reeb space is a polyhedron (\cite{shiota}).
\end{enumerate}
\end{Ex}
\subsection{Facts on immersions and embeddings of standard spheres.}
\begin{Prop}
\label{prop:2}
Let $k$ be a non-negative integer. For the canonical embedding $i_0:S^k \subset \mathbb{R}^{k+1}$ where $S^k$ is the unit sphere, for any orientation-preserving diffeomorphism $\phi$, $i_0$ and $i_0 \circ \phi$ are regularly homotopic. if $k=0,2,6$ holds, then for any diffeomorphism $\phi$, they are regularly homotopic.
\end{Prop}
\begin{proof}
The former is presented in \cite{saekitakase} with an explanation of the basic idea \cite{kaiser} and the latter follows from the fact that we can reverse every unit sphere
 of dimension $0$,  $2$ or $6$ by a regular homotopy in the outer Euclidean space of dimension larger than the sphere by $1$ (\cite{smale}):
 considering the operation of this reversing, we
 represent the original orientation reversing diffeomorphism $\phi$ as a composition of a corresponding appropriate orientation
 reversing diffeomorphism and an orientation preserving one and this completes the proof.
\end{proof}

For embeddings, we can know the following presented in \cite{saekitakase} for example.

\begin{Prop}
\label{prop:3}
Let $k$ be a non-negative integer not $4$. For the caonical embedding $i_0:S^k \subset \mathbb{R}^{k+1}$ where $S^k$ is the unit sphere, for an orientation preserving diffeomorphism $\phi$ on $S^k$, $i_0$ and $i_0 \circ \phi$ are isotopic if and only if $\phi$ is isotopic to the identity.
Especially, for $0 \leq k \leq 5$ with $k \neq 4$, an orientation preserving diffeomorphism $\phi$ always satisfies the condition. Moreover, even if $k=4$ holds, $i_0$ and $i_0 \circ \phi$ are always isotopic.
\end{Prop}

\subsection{Cerf's theory, a study by Wrazidlo and basic tools.}
We introduce a fundamental result based on Cerf's theory \cite{cerf}, which is used in \cite{saeki2} for example.
\begin{Prop}
\label{prop:4}
Let $f:M \rightarrow \mathbb{R}$ a special generic function on a homotopy sphere of dimension not $5$. 
We can smoothly isotope any orientation preserving diffeomorphism $\phi$ on $M$ to ${\phi}_0$ so that $f \circ {\phi}_0=f$ holds.
\end{Prop}

We introduce works of Wrazidlo \cite{wrazidlo}.
\begin{Def}[\cite{wrazidlo}]
\label{def:5}
A special generic map on a closed manifold is said to be {\it standard} if the Reeb space is a closed disc of dimension equal to the dimension of the target manifold. 
\end{Def}

Standard special generic maps are homotopic to a special generic map whose singular set is a sphere and whose restriction to the singular set is an embedding. Moreover, we can consider this homotopy as a homotopy such that at each point, the map is also a standard special generic map. 
Note that if a closed manifold admits a standard special generic map into the $n$-dimensional Euclidean space, it admits a standard special generic
 map into any lower dimensional
 Euclidean space. We can know this by considering a canonical projection or more generally, by virtue of Proposition \ref{prop:7} or Corollary \ref{cor:1} later. In fact, if we project the unit disc representing a standard special generic map by a canonical projection, the image is also the unit disc
 in the target space and the singular value set is the boundary of the unit disc. We define the {\it fold perfection} of a homotopy sphere by
 the maximal dimension of the Euclidean space into which the sphere admits a standard special generic map. We define the {\it special generic dimension set} of a closed manifold as the set of all integers $k$ such that the manifold admits a special generic map into ${\mathbb{R}}^k$: in Example \ref{ex:2.2}, a special generic map on a homotopy sphere into a Euclidean whose dimension is lower than the source manifold is characterized as a special generic map into a Euclidean whose dimension is lower than the source manifold such that the Reeb space is contractible. We review several explicit related facts.

\begin{Ex}
\label{ex:3}
\begin{enumerate}
\item
A Morse function with just two singular points before is a standard special generic map.
\item Canonical projections of unit spheres are standard special generic maps. 
\item A closed manifold admitting a standard special generic map is a homotopy sphere and if the dimension of the source manifold is smaller than $7$, then the homotopy sphere is a standard sphere. Moreover, every homotopy sphere except exotic $4$-dimensional spheres admits a
 standard special generic map into the plane and the fold perfection is larger than $1$. See \cite{saeki2} for example. 
\end{enumerate}
\end{Ex}

\begin{Def}
\label{def:6}
Let $k$ be an integer larger than $6$. Let $l$ be a positive integer not larger than $k$ and let a diffeomorphism $\phi$ on the unit disc $D^{k-1} \subset {\mathbb{R}}^{k-1}$ fixing all the points in the boundary can be smoothly isotoped to ${\phi}_0$
so that ${{\pi}_{k-1,l-1}} \mid_{D^{k-1}} \circ {\phi}_0={\pi}_{k-1,l-1} {\mid}_{D^{k-1}}$. The {\it Gromoll filtration number} of the diffeomorphism $\phi$ is defined as the maximal number $l$.  
\end{Def}

\begin{Ex}
\label{ex:4}
For an integer $k>6$ and for
 every diffeomorphism on the unit disc $D^{k-1} \subset {\mathbb{R}}^{k-1}$ fixing all the points in the boundary, the Gromoll filtration number is larger than $1$. 
\end{Ex}

\begin{Prop}[\cite{wrazidlo}]
\label{prop:5}
Let $k \geq 6$ be an integer.
We can obtain an orientation preserving diffeomorphism on $S^k$ by regarding $D^k$ as the hemisphere of the unit sphere of ${\mathbb{R}}^{k+1}$ by an injection $x \mapsto (x,\sqrt{1-{\parallel x \parallel}^2}) \in S^k$ and extending
 the diffeomorphism fixing all the points of the boundary on $D^k$ to the unit sphere $S^k$ by using the identity map. If 
the Gromoll filtration number is larger than $l$, then we can smoothly isotope the diffeomorphism to ${\phi}_0$ so that ${\pi}_{k+1,l} {\mid}_{S^k} \circ {\phi}_0 = {\pi}_{k+1,l} {\mid}_{S^k}$
holds.
For a positive integer $k$ smaller than $6$ and not $4$ and a positive integer $l \leq k$, then we can smoothly isotope the diffeomorphism on $S^k$ to ${\phi}_0$ so
 that ${\pi}_{k+1,l} {\mid}_{S^k} \circ {\phi}_0 = {\pi}_{k+1,l} {\mid}_{S^k}$ holds.
\end{Prop}

Note that Example \ref{ex:4} and Proposition \ref{prop:5} prove Proposition \ref{prop:4}.

The following is a key to the proof of Theorem \ref{thm:3}, extending Proposition \ref{prop:6}.

\begin{Prop}
\label{prop:6}
If $k$ is $1,2,3$ or $5$, or for an integer $k \geq 6$ and the Gromoll filtration number of the diffeomorphism on $D^k$ fixing all the points in the boundary is always larger than $l$, then we can smoothly isotope
 any diffeomorphism $\phi$ on $S^k$ to ${\phi}_0$ so that ${\pi}_{k+1,l} {\mid}_{S^k} \circ {\phi}_0 = {\pi}_{k+1,l} {\mid}_{S^k}$ holds.
\end{Prop}
\begin{proof}
Every orientation reversing diffeomorphism on $S^k$ regarded as the unit sphere $S^k \subset {\mathbb{R}}^{k+1}$ is represented by the composition of an orientation reversing diffeomorphism on $S^k$ and an orientation preserving diffeomorphism
 on $S^k$, which is obtained as the extension of a diffeomorphism on the disc $D^k$ fixing all the points of the boundary regarded as the hemisphere of $S^k$, by virtue of Proposition \ref{prop:5}.
In the decomposition of the diffeomorphism, we can take any orientation reversing one and take such a diffeomorphism ${\phi}^{\prime}$ so that the relation ${\pi}_{k+1,l} {\mid}_{S^k} \circ {{\phi}^{\prime}} = {\pi}_{k+1,l} {\mid}_{S^k}$ holds: for example consider a reflection on a hyperplane including the origin and its reflection to the unit sphere. From this, we can take the orientation reversing diffeomorphism $\phi$ so that ${\pi}_{k+1,l} {\mid}_{S^k} \circ {\phi} = {\pi}_{k+1,l} {\mid}_{S^k}$ holds. This completes the proof.
\end{proof}

\subsection{Results.}
\subsubsection{Lifts to immersions.}
\begin{Thm}
\label{thm:1}
We can lift a normal standard-spherical Morse function to an immersion of codimension $1$ if the source manifold $M$ of dimension $m \geq 1$ is orientable.
If $m=1,3,7$ holds, then we do not need the assumption of the orientability of $M$. 
Moreover, we can lift a normal standard-spherical Morse function to an embedding of codimension $1$ if the source manifold $M$ of dimension $m\geq 1$ is orientable and for the dimension $m$, either of the following holds.
\begin{enumerate}
\item $1 \leq m \leq 6$.
\item $m \geq 7$ and the oriented diffeomorphism group of $S^{m-1}$ is connected.
\end{enumerate}
\end{Thm}
\begin{proof}
The inverse image of a small closed neighborhood of a point of the Reeb space whose inverse image contains a singular point can be lifted to a codimension $1$ embedding
 as in FIGURE \ref{fig:1} or FIGURE \ref{fig:2}. 

\begin{figure}
\includegraphics[width=10mm]{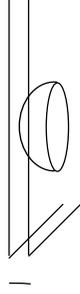}
\caption{The inverse image of a small $1$-dimensional subcomplex including a singular point of index $0$.}
\label{fig:1}
\end{figure} 
\begin{figure}
\includegraphics[width=20mm]{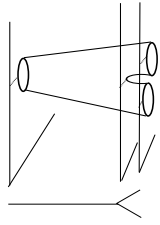}
\caption{The inverse image of a small $1$-dimensional subcomplex including a singular point of index $1$.}
\label{fig:2}
\end{figure} 
They are regarded as a subset of the unit sphere and a product of two standard spheres canonically embedded as
 submanifolds whose codimensions are $1$, respectively. For the former case, the fact follows from the fact that the function is locally regarded as the natural height function of a closed disc.  Note also that for the latter case, we need a specific case of Lemma \ref{lem:1}.

\begin{Lem}
\label{lem:1}
For a pair of small closed and connected subspaces of Reeb spaces of normal spherical Morse functions such that each subspace includes just one branched point, consider the inverse images and the restrictions of the functions. We assume that the following hold.
\begin{enumerate}
\item For the images of the restrictions of the functions, consider the inverse images of points in the boundaries such that the inverse images are connected, then they are diffeomorphic homotopy spheres.   
\item For the images of the restrictions of the functions, consider the inverse images of points in the boundaries such that the inverse images are not connected {\rm (}and have $2$ connected components diffeomorphic to homotopy spheres{\rm )}, then they are diffeomorphic.    
\end{enumerate}
Then the restrictions of the two functions, which we denote by $c_1$ and $c_2$ respectively, are $C^{\infty}$ equivalent or for two suitable diffeomorphisms $\Phi$ and $\phi$, the relation $\phi \circ c_1=c_2 \circ \Phi$ holds. 
\end{Lem}

\begin{proof}[Sketch of the proof]
This is mentioned in section 2 of \cite{saeki4} etc. in the case where the inverse image is $2$-dimensional. 
We present a sketch of the proof based on FIGURE \ref{fig:a}, representing the source manifold of a restriction we are considering of a normal spherical Morse function.

\begin{figure}
\includegraphics[width=100mm]{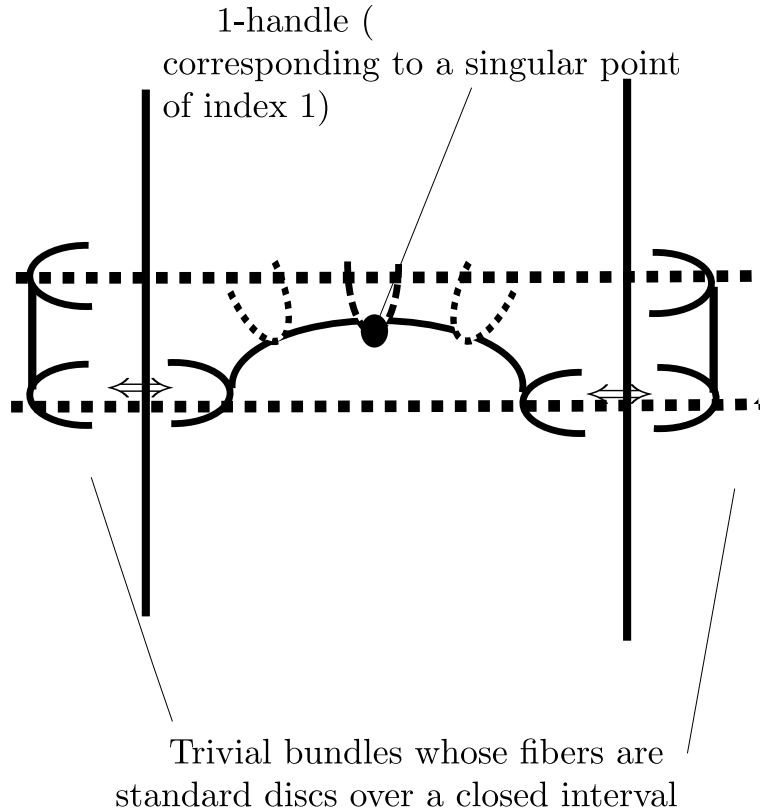}
\caption{The restriction of a normal spherical function to an inverse image of a small $1$-dimensional subcomplex including a singular point of index $1$.}
\label{fig:a}
\end{figure} 

The obtained function is decomposed into $3$ functions as the figure, One is a restriction of a Morse function naturally corresponding to a $1$-handle and the others have no singular points and are trivial bundles over the closed interval coinciding with the image by virtue of (the relative version of) Ehresmann's fibration theorem. By gluing them in a canonical way, we obtain the function and the two constraints on smooth structures of homotopy spheres of inverse images of regular values complete the proof. 
\end{proof}

Note that by canonical projections, we obtain a cobordism of special generic functions and maps, discussed in \cite{saeki3} and also
 in \cite{sadykov} and this fact is key in Theorem \ref{thm:3} later for example. The image is as in FIGURE \ref{fig:3} and the thick lines except vertical ones represent the singular value set, which is an embedded image of the singular set.

\begin{figure}
\includegraphics[width=40mm]{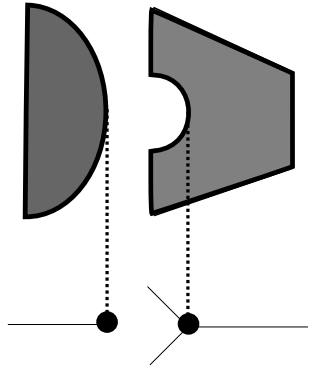}
\caption{Cobordisms of special genreic functions (maps).}
\label{fig:3}
\end{figure} 

Then, the rest of the proof is essentially based on the idea of the proof of Theorem 4.1 of \cite{saekitakase}. We can lift the inverse images of subsets
 PL homeomorphic to closed intervals consisting of non-singular points to immersions obtained by considering regular
 homotopies of immersions of spheres by virtue of Proposition \ref{prop:2}. As a result, we obtain the desired lift.

 The rest of the statement follows from Proposition \ref{prop:3} by considering isotopies instead of regular homotopies.
\end{proof}

We will present visualizations of explicit lifts of normal spherical Morse functions of Theorem \ref{thm:1} with and Theorem \ref{thm:2} and Theorem \ref{thm:3}, presented later, in Example \ref{ex:5}.

\begin{Thm}
\label{thm:2}
We can lift every normal standard-spherical Morse function to an immersion of codimension $2$. 
\end{Thm}
\begin{proof}
As Theorem \ref{thm:1}, the inverse image of a small closed neighborhood of a point of the Reeb space whose inverse image contains a singular point can be lifted to a codimension $2$ embedding
 as in FIGURE \ref{fig:1} or FIGURE \ref{fig:2} before.
 
Then, the rest of the proof is essentially based on the ideas of the proofs of Theorem 4.1 \cite{saekitakase} and Theorem \ref{thm:1}. By virtue of Proposition \ref{prop:2}, we can lift the inverse images of subsets
 PL homeomorphic to closed intervals consisting of non-singular points to immersions obtained by considering regular
 homotopies of immersions of spheres preserving the orientation and if we need, by considering also an isotopy as in FIGURE \ref{fig:4} for reversing the orientation of the original image of the
 standard embedding $S^{m-1} \subset {\mathbb{R}}^m \subset {\mathbb{R}}^{m+1}$. Note that in the situation of this theorem including the previous embedding, we regard $S^{m-1}$ as the unit sphere of ${\mathbb{R}}^m$ and
 ${\mathbb{R}}^m$ as ${\mathbb{R}}^m \times \{0\} \subset {\mathbb{R}}^m \times \mathbb{R}={\mathbb{R}}^{m+1}$ canonically. We can know that for the standard
 embedding $i:S^{m-1} \rightarrow {\mathbb{R}}^{m+1}$ of codimension $2$ just before and for any diffeomorphism $\phi$ on $S^{m-1}$, $i$ and $i \circ \phi$ are regularly homotopic. As
 a result, we obtain the desired lift.
\begin{figure}
\includegraphics[width=30mm]{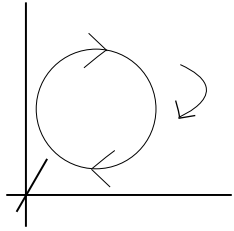}
\caption{An isotopy for reversing the orientation of the original image of the standard embedding of a unit sphere of codimension $2$ by turning the sphere upside-down.}
\label{fig:4}
\end{figure} 
\end{proof}
\subsubsection{Lifts to immersions, embeddings, and special generic maps.}
We review construction of spherical maps by \cite{saekisuzuoka}. It is explicitly used in the present paper in considerable cases. For a smooth map $c$, we denote the set of all singular points ({\it singular set}) by $S(c)$. we call $c(S(c))$ the {\it singular value set} of $c$. 
\begin{Prop}[\cite{saekisuzuoka} etc.]
\label{prop:7}
Let $m$, $n$ and $n^{\prime}$ be positive integers satisfying the relation $m \geq n>n^{\prime}$.
If we compose the canonical projection ${\pi}_{n{,n}^{\prime}}$ to a proper stable special generic map from an $m$-dimensional closed or open manifold into the $n$-dimensional Euclidean space into a lower dimensional Euclidean space ${\mathbb{R}}^{n^{\prime}}$, then
 it is a normal spherical fold map if the following hold.
\begin{enumerate}
\item ${{\pi}_{n,n^{\prime}} \circ f} \mid_{S(f)}$ is a fold map. 
\item For every $y \in ({\pi}_{n,n^{\prime}} \circ f)(M)- ({\pi}_{n,n^{\prime}} \circ f)(S(({\pi}_{n,n^{\prime}} \circ f) {\mid}_{S(f)}))$ each connected component 
$({\pi}_{n,n^{\prime}} \circ \bar{f})^{-1}(y) \subset W_f$ is compact and contractible.
\end{enumerate}
\end{Prop}

\begin{Cor}[\cite{saekisuzuoka}]
\label{cor:1}
Let $n \geq 2$.
If we project a proper stable special generic map $f:M \rightarrow {\mathbb{R}}^n$ into ${\mathbb{R}}^{n-1}$ such that the map ${{\pi}_{n,n-1} \circ f} \mid_{S(f)}$ is a fold map and that $({\pi}_{n,n-1} \circ \bar{f})^{-1}(y) \subset W_f$ is compact, then
 it is a normal spherical fold map. 
\end{Cor}
\begin{Rem}
\label{rem:1}
Proposition \ref{prop:7} and Corollary \ref{cor:1} for the case where the source manifold is open is for the case where the map is obtained as the restriction of a cobordism of a special generic map to the interior of the compact manifold giving the cobordism for example. 
\end{Rem}
In this paper, conversely, we consider lifting normal spherical fold maps to special generic maps as mentioned in Section \ref{sec:1}. Before the presentation of a result, we review a fundamental result by \cite{saeki2}.

\begin{Prop}
\label{prop:8}
A closed manifold $M$ of dimension $m>2$ admits a special generic map into the plane if and only if $M$ is represented as the following.
\begin{enumerate}
\item In the case where $m$ is smaller than $7$, $M$ is represented as a connected sum of finite copies of $S^1 \times S^{m-1}$ and total spaces of non-orientable $S^{m-1}$-bundles over $S^1$.
\item  In the case where $m$ is not smaller than $7$, $M$ is represented as a connected sum of a finite number of manifolds represented as total spaces of trivial bundles whose fibers are homotopy spheres over $S^1$ or total spaces of non-orientable bundles whose fibers are homotopy spheres over $S^1$.
\end{enumerate}
\end{Prop}

We have the following.
\begin{Thm}
\label{thm:3}
\begin{enumerate}
\item
\label{thm:3.1}
 Every normal spherical Morse function on a closed manifold of dimension $m>1$ and $m \neq 5$ can be lifted to a special generic map into the plane.
\item
\label{thm:3.2}
 Every normal standard-spherical Morse function on a closed manifold of dimension $m>1$ and $m \neq 5$ can be lifted to
 a special generic map into ${\mathbb{R}}^n$ satisfying $m \geq n \geq 2$ if $m \leq 6$ holds or $m \geq 7$ holds and the Gromoll filtration number of every diffeomorphism on $D^{m-1}$ fixing all the points in the boundary is larger
 than $n-1$. 
\end{enumerate}
\end{Thm}
\begin{proof}
First we consider general integers satisfying $m \geq n \geq 2$. In the case where inverse images of regular values are always disjoint unions of standard spheres there, the inverse image of a small closed neighborhood of a point of the Reeb space whose inverse image contains a singular point can be lifted to a cobordism of special generic functions as presented and to a cobordism of standard special generic maps
 into ${\mathbb{R}}^{n-1}$ so that all the special generic maps on the boundaries are all regarded as the canonical projections of unit spheres in ${\mathbb{R}}^{m}$ into ${\mathbb{R}}^{n-1}$ as presented in FIGURE \ref{fig:3}, which is stated also in the presentation of FIGURE \ref{fig:1}, FIGURE \ref{fig:2} and FIGURE \ref{fig:3} before.
Even in the case where there exists a connected component of a regular value not diffeomorphic to the standard sphere, we can lift such a local function to a cobordism of special generic functions same as the previous one presented as FIGURE \ref{fig:3} in the PL (topology) category (see also \cite{saeki3} for example for precise theory). 
Note that in constructing lifts in these cases, Lemma \ref{lem:1} plays important roles and Proposition \ref{prop:7} (Corollary \ref{cor:1}) and Remark \ref{rem:1} for proper maps are also important .

We prove the former statement first.

We lift the inverse images of subsets
 PL homeomorphic to closed intervals consisting of points being not singular to products of the special generic functions and the identity maps on closed intervals.
We can glue them so that we can obtain a special generic map being a lift of the original function by virtue of Cerf's theory or Proposition \ref{prop:4}: see also FIGURE \ref{fig:5}.

The proof of the latter one is similar to the proof of the former one. We lift the inverse image of a small closed neighborhood of a point of the Reeb space whose inverse image contains a singular point to a
 cobordism of standard special generic maps as mentioned. We lift the inverse images of subsets
 PL homeomorphic to closed intervals consisting of points being not singular to products of the special generic maps regarded as canonical projections of unit spheres and the identity maps on closed intervals as a higher dimensional version of the discussion before.
 We can glue the lifts by virtue of Proposition \ref{prop:6} to obtain a desired lift: see also FIGURE \ref{fig:5}. 

This completes the proof.

\begin{figure}
\includegraphics[width=40mm]{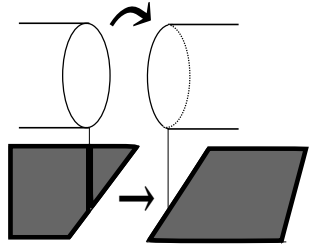}
\caption{Gluing given two special generic functions or standard special generic maps on homotopy spheres appearing in the inverse images of regular values.}
\label{fig:5}
\end{figure}
\end{proof}

\begin{Ex}
\label{ex:5}
\begin{enumerate}
\item FIGURE \ref{fig:6} represents lifts of explicit normal spherical Morse functions to immersions (special generic maps)

\begin{figure}
\includegraphics[width=40mm]{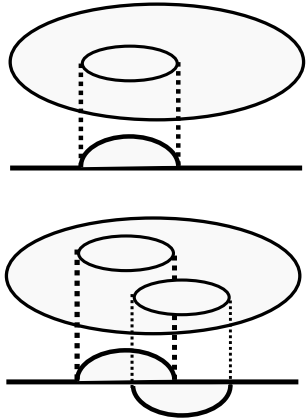}
\caption{Lifts of explicit normal spherical Morse functions to immersions (special generic maps).}
\label{fig:6}
\end{figure} 
\item FIGURE \ref{fig:7} represents a lift of an explicit normal spherical Morse function to an immersion (a special generic map); this shows an example such that distinct connected components of the inverse image of a regular value may intersect in the target space where the lift is a special generic map into the plane. 
\begin{figure}
\includegraphics[width=40mm]{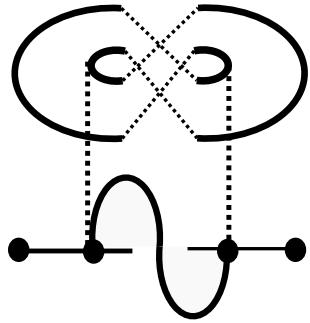}
\caption{A lift of an explicit normal spherical Morse function to an immersion (a special generic maps).}
\label{fig:7}
\end{figure} 
\end{enumerate}
\end{Ex}

\begin{Ex}
\label{ex:6}
According to \cite{crowleyschick} for example, the number of connected components of the orientation preserving diffeomorphism group of $S^{11}$ and the group of all the diffeomorphisms on $D^{11}$ fixing all the points in the boundary are both $1$. We can
 apply Theorem \ref{thm:3} (\ref{thm:3.2}) for $1 \leq n \leq 12$.
 Moreover, for every orientation preserving diffeomorphism of $D^{12}$, the Gromoll filtration number is always larger than $3$. The number of connected components of the orientation preserving diffeomorphism group of $S^{12}$ and the group of all the diffeomorphisms on $D^{12}$ fixing all the points in the boundary are both $3$.  We can
 apply Theorem \ref{thm:3} (\ref{thm:3.2}) for $1 \leq n \leq 4$ and by the fact before, we can drop the assumption "standard-spherical" or in this case, "spherical" means "standard-spherical". 
 \end{Ex}

\begin{Rem}
\label{rem:2}
The attachments in the proof of the former statement are also done in \cite{saeki2} to construct special generic maps into the plane on several closed manifolds. However, only some orientation
 reversing diffeomorphisms are considered, and projections of the resulting special generic maps to $\mathbb{R}$ are not considered. A similar method used in the latter statement is a general
 explicit case of the attachements Wrazidlo \cite{wrazidlo} performed to construct standard special generic maps on homotopy spheres based on Proposition \ref{prop:5}.
\end{Rem}

We have the following stating that some of normal spherical Morse functions cannot be lifted to special generic maps into some Euclidean spaces.

\begin{Thm}
\label{thm:4}
If for a normal spherical Morse function on a closed manifold of dimension $m>1$, there exists an inverse image of a regular value containing a homotopy sphere such that $n-1>0$ is not in the special generic dimension set of the sphere, then it can not be lifted to a special generic map into $\mathbb{R}^n$. 
 In addition, if for a normal spherical Morse function, there exists an inverse image of a regular value containing a homotopy sphere such that the fold perfection
 is smaller than $3$, then it cannot be lifted to a special generic
 map into ${\mathbb{R}}^4$.
\end{Thm}

%

\begin{proof}[Proof of Theorem \ref{thm:4}]
For a normal spherical Morse function, let there exist an inverse image of a regular value containing a homotopy sphere such that $n-1>0$ is not in the special generic dimension set of the sphere. We assume that it can be lifted to a special generic map into ${\mathbb{R}}^n$. Consider the restriction of the special generic map to the inverse image of a regular value of the Morse function. It is regarded as a special generic map into ${\mathbb{R}}^{n-1}$. This contradicts the assumption. 
We show the latter. Let there exist an inverse image of a regular value containing a homotopy sphere such that $3$ is larger than the fold perfection of the sphere. This means that the homotopy sphere admits a standard special generic map into ${\mathbb{R}}^{3}$, which is a contradiction. Such a homotopy sphere does not admit a special generic map into ${\mathbb{R}}^3$; if it admits a special generic map into ${\mathbb{R}}^3$, then the Reeb space is a contractible $3$-dimensional compact manifold and it is a $3$-dimensional standard closed disc. This completes the proof.
\end{proof}
\begin{Ex}
\label{ex:7}
According to \cite{saeki2}, for any standard sphere, the fold perfection is equal to the dimension of the sphere and for any homotopy sphere of dimension $k>6$ not diffeomorphic to $S^k$, the fold perfection must be smaller than $k-3$. Thus we can easily construct a normal spherical Morse function on manifold of dimension $m>7$ which cannot be lifted to a special generic map into ${\mathbb{R}}^{m-3}$. 
According to \cite{wrazidlo}, there are at least $14$ oriented diffeomorphism types of $7$-dimensional oriented homotopy spheres such that the fold perfections are $2$. Thus we can easily construct a normal spherical Morse function on an $8$-dimensional manifold which cannot be lifted to a special generic map into ${\mathbb{R}}^4$. For example, in the first figure of FIGURE \ref{fig:6}, we can construct a spherical fold map whose Reeb space is as presented such that the inverse image of a regular value in the center is a disjoint union of two copies of such an exotic homotopy sphere $\Sigma$; $S^1 \times \Sigma$ admits such a function and it is also obtained by projecting a special generic map into the plane regarded as the product of the identity map on the circle and a special generic function on the homotopy sphere (implicitly appearing as special generic maps into the plane in Proposition \ref{prop:8}) canonically by applying Proposition \ref{prop:7}. 
\end{Ex}

As a key proposition, we introduce a result in \cite{saekitakase}.
\begin{Prop}[\cite{saekitakase}]
\label{prop:9}
\begin{enumerate}
\item A special generic map on a closed manifold into the plane can be lifted to a codimension $1$ immersion if the source manifold is orientable.
\item A special generic map on a non-orientable closed manifold into the plane can be lifted to a codimension $1$ immersion if and only if the 
difference $m-2$ between the dimension of the source manifold and that of the target manifold is $0$, $2$ or $6$ and the normal bundle of the singular
 set is orientable and trivial.
\item A special generic map on a closed manifold of dimension $m>1$ into the plane can be lifted to a codimension $1$ embedding if and only if the source manifold is represented as the connected sum of a finite number of $S^1 \times S^{m-1}$.
\end{enumerate}
\end{Prop}
Using this together with Theorem \ref{thm:1} and Theorem \ref{thm:3}, we have the following.
\begin{Thm}
\label{thm:5}
\begin{enumerate}
\item A normal spherical Morse function on a closed orientable manifold of dimension can be lifted to an immersion of codimension $1$.
\item
For any codimension $1$ immersion obtained as a lift of the original normal {\rm (}standard-{\rm )}spherical Morse function on a closed non-orientable manifold of dimension $m=3,7$ and  in Theorem \ref{thm:1}, if we compose it with ${\pi}_{m+1,2}$, then the resulting
 map never be special generic.
\item Every normal spherical Morse function on a closed manifold of dimension $m>1$ represented as the connected sum of a finite number of $S^1 \times S^{m-1}$ can be lifted to an embedding of codimension $1$.
\end{enumerate}
\end{Thm}

\begin{proof}
The first statement follows from Theorem \ref{thm:1}, the first statement of Theorem \ref{thm:3} and the first statement of Proposition \ref{prop:9}. The second statement follows from the second statement of Proposition \ref{prop:9}. The last
 statement follows from the first statement of Theorem \ref{thm:3} and the last statement of Proposition \ref{prop:9} in the case where $m \neq 5$ holds and from Theorem \ref{thm:1} in the case where $m=5$ holds.  
\end{proof}

From Proposition \ref{prop:9}, Corollary \ref{cor:1} and additional discussions, we have the following.
\begin{Thm}
\label{thm:6}
 Let $m=2,4,8$. Consider a special generic map obtained as a lift of a normal standard-spherical Morse function on a closed non-orientable manifold of dimension $m$ and an immersion lift in the second statement of Proposition \ref{prop:8}. Thus we obtain a codimension $1$ immersion lift of the Morse function and for a regular value $a$, there exists a connected component of the inverse image and
 the immersion of the component into ${{\pi}_{m+1,1}}^{-1}(a)$ is not regularly homotopic to the canonical embedding of the unit sphere. 
\end{Thm}
\begin{proof}
Assume that for no connected component of the inverse image of each regular value $a$, the immersion of the component into ${{\pi}_{m+1,1}}^{-1}(a)$ is not regularly homotopic to the canonical embedding of the unit sphere. Thus, the assumption $m-1=1,3,7$ means that the source manifold is orientable by \cite{smale}. This contradicts the assumption that the source manifold is non-orientable. This completes the proof.
\end{proof}

Related to Proposition \ref{prop:9}, we obtain the following about the normal bundles of singular sets of special generic maps obtained as lifts
 of normal spherical Morse functions.

\begin{Thm}
\label{thm:7}
Let $m>n \geq 2$ be positive integers.
Assume that we can lift a normal spherical Morse function on a closed manifold of dimension $m$ to a special generic map whose Reeb space is a handlebody or a compact manifold obtained by attaching a finite number of $1$-handles to $D^n$ and smoothing. Then the following hold.
\begin{enumerate}
\item If $n \geq 3$ holds and the source manifold of the function is non-orientable, then
 the normal bundle of the singular set of the resulting special generic map is not orientable.
\item For $2 \leq n \leq 4$, if the source manifold of the function is orientable, then
 the normal bundle of the singular set of the resulting special generic map is trivial.
\end{enumerate}
\end{Thm}
\begin{proof}
We prove the first part. The boundary of the handlebody is represented as a connected
 sum of a finite number of $S^1 \times S^{n-2}$. The 1st homology group of the
 handlebody whose coefficient ring is $\mathbb{Z}/2\mathbb{Z}$ is generated by the direct sum of a finite number of $H_1(S^1 \times \{0\};\mathbb{Z}/2\mathbb{Z})$, isomorphic to $\mathbb{Z}/2\mathbb{Z}$. The normal bundle of the singular set is a linear bundle over the boundary and the non-orientability of the source manifold and Example \ref{ex:2} (\ref{ex:2.2}) or (\ref{ex:2.3}) together with the assumption $m-1>m-n \geq 1$ for example, imply that the 1st Stiefel-Whitney class of the normal bundle does not vanish. \\
We prove the second part. The normal bundle must be orientable since the source manifold is orientable and the base space of the bundle is orientable. Moreover, the normal bundle is spin since by the following reasons, the 2nd Stiefel-Whitney classes vanish.
\begin{enumerate}
\item In the case where $n=2$ holds, the source manifold is as in Proposition \ref{prop:8} and spin.
\item In the case where $n=3,4$ holds, $m \geq 4$ holds and from Example \ref{ex:2} (\ref{ex:2.3}), $m-1 \geq 3$ holds and the 2nd homology groups vanish.
\end{enumerate}

Note that there may be other explanations. Moreover, by a fundamental fact on the Stiefel-Whitney classes of linear bundles, the 3rd Stiefel-Whitney class vanishes. Furthermore, the dimension of the base space is smaller than $4$. This means that the normal bundle is trivial.
\end{proof}

The following is a result on lifts of special generic maps to embeddings.

\begin{Prop}[\cite{nishioka2}]
\label{prop:10}
For a special generic map $f:M  \rightarrow {\mathbb{R}}^n$ of a closed orientable manifold $M$ of dimension $m>n$. If for an integer $k$, $k \geq \max\{\frac{3m+3}{2},m+n+1\}$ holds and the
normal bundle of the singular set of $f$ is trivial, then $f$ can be lifted to an embedding into ${\mathbb{R}}^k$.  
\end{Prop}
By virtue of this together with Theorem \ref{thm:1} and Theorem \ref{thm:7}, we immediately have the following.
\begin{Thm}
\label{thm:8}
Every normal spherical Morse function on a closed orientable manifold of dimension $m \geq 2$ satisfying $m \neq 5$ can be lifted to an embedding into $k$-dimensional Euclidean space where $k \geq \max\{\frac{3m+3}{2},m+2+1\}$ holds so that the composition
 of the embedding and ${\pi}_{k,2}$ is a special generic map into ${\mathbb{R}}^2$ and that the composition of the special generic map and ${\pi}_{2,1}$ is the original Morse function.
Furthermore, if $m<7$ or $m \geq 7$ and the Gromoll filtration number of all the diffeomorphisms on $D^{m-1}$ fixing all the points in the boundary is larger than $n-1$ where $n=3,4$, then every normal standard-spherical Morse function on a closed orientable manifold of dimension $m>n$ can be lifted to an embedding into ${\mathbb{R}}^k$ so that
the composition
 of the embedding and ${\pi}_{k,n}$ is a special generic map into ${\mathbb{R}}^n$ and that the composition of the special generic map and ${\pi}_{n,1}$ is the original Morse function.
\end{Thm}

Moreover, as an example together with Example \ref{ex:6}, we have the following.
\begin{Cor}
Every normal spherical Morse function on a closed orientable manifold of dimension $13$ can be lifted to an embedding into $21$-dimensional Euclidean space so that the composition
 of the embedding and ${\pi}_{21,2}$ is a special generic map into ${\mathbb{R}}^2$ and that the composition of the special generic map and ${\pi}_{2,1}$ is the original Morse function.
Furthermore, for $n=3,4$, every normal spherical Morse function on a closed orientable manifold of dimension $13$ can be lifted to an embedding into ${\mathbb{R}}^{21}$ so that
the composition
 of the embedding and ${\pi}_{21,n}$ is a special generic map into ${\mathbb{R}}^n$ and that
 the composition of the special generic map and ${\pi}_{n,1}$ is the original Morse function.
\end{Cor}
\begin{Rem}
\label{rem:3}
Related to the problems here, we end the present paper by introducing another explicit question. 

There are plenty of pairs of special generic maps into different target manifolds or a fixed target manifold which are not $C^{\infty}$ equivalent on a fixed closed manifold. We can see such facts in \cite{nishioka} together with \cite{saeki2} for example: Nishioka has constructed special generic maps on manifolds represented as connected sums of finite numbers of $S^2 \times S^3$ or total spaces of non-trivial linear $S^3$-bundles over $S^2$ into ${\mathbb{R}}^4$, where Saeki constructed such maps on these manifolds into ${\mathbb{R}}^3$. 

As another study, Yamamoto \cite{yamamoto} studied $C^{\infty}$ equivalence of special generic maps on homotopy spheres into the plane: for example, Yamamoto has explicitly shown that there exist homotopy spheres admitting two $C^{0}$ equivalent standard special generic maps into the plane being not $C^{\infty}$ equivalent. 

We explain more precisely about the study of Yamamoto. He has shown that on a homotopy sphere $\Sigma$ of dimension larger than $2$, there exist special generic maps $f_1:\Sigma \rightarrow {\mathbb{R}}^2$ and $f_2:\Sigma \rightarrow {\mathbb{R}}^2$ such that a homeomorphism $\Phi$ satisfying the relation $f_1=\Phi \circ f_2$ exists and that we cannot take $\Phi$ as a diffeomorphism. 
 $f_1$ and $f_2$ are lifts of a special generic function on $\Sigma$ and we can also obtain lifts as immersions or embeddings of $f_1$ and $f_2$ by virtue of the study of the present paper. However, we do not know the immersions or embeddings obtained as lifts are equivalent under appropriate relations such as $C^{\infty}$ equivalence, isotopy, regular homotopy etc..
\end{Rem}
Related to Remark \ref{rem:3}, as an appendix, we present an example of a pair of normal spherical Morse functions on a fixed manifold which are not {\it regularly equivalent}: Two such functions $f_1:M \rightarrow \mathbb{R}$ and $f_2:M \rightarrow \mathbb{R}$ are said to be {\it regularly equivalent} if there exists an diffeomorphism $\Phi$ on $M$ (smoothly) isotopic to the identity map on $M$ satisfying the relation $f_1=f_2 \circ \Phi$ and such equivalence classes are discussed in \cite{yamamoto}, which are essentially equivalent to $C^{\infty}$ equivalence classes in the case where we only consider special generic functions and standard special generic maps into the plane (whose singular value sets are embedded circles).

We consider a normal spherical Morse function which is not special generic and consider the Reeb space. We can change diffeomorphisms representing identifications at the inverse images of two points near points corresponding to the values of singular points whose indices are $0$. If the orientation preserving diffeomorphism group of the standard sphere appearing as the inverse is not connected and the dimension of the source manifold is larger than $6$, then we can construct a new function such that the original function and the new one is not regularly equivalent and if we take a diffeomorphism ${\Phi}_1$ arbitrary and the other diffeomorphism ${\Phi}_2$ suitably, then we obtain a new source manifold diffeomorphic to the original one. This procedure gives an operation to the source manifold as the following.
\begin{enumerate}
\item We take a connected sum of the original source manifold and a homotopy sphere.
\item We take a connected sum of the obtained source manifold and another homotopy sphere to obtain a source manifold diffeomorphic to the original one.
\end{enumerate}
To obtain a map such that the equivalence class is different from that of the original one, we take ${\Phi}_1$ as a diffeomorphism not smoothly isotopic to the original one and the diffeomorphism obtained by the composition of the original diffeomorphism and ${{\Phi}_1}^{-1}$ is orientation preserving. Note that the two maps are $C^{0}$ equivalent. 
\begin{figure}
\includegraphics[width=30mm]{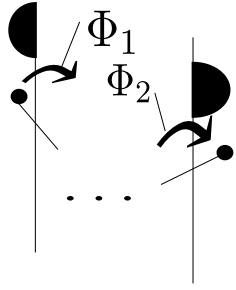}
\caption{We can change the diffeomorphisms ${\Phi}_1$ and ${\Phi}_2$ on the spheres appearing as inverse images to glue local maps together.}
\label{fig:8}
\end{figure}
See also FIGURE \ref{fig:8}.

Last, if we lift functions of such a pair as Theorem \ref{thm:3}, then we obtain special generic maps into the plane and in suitable cases, ones into higher dimensional spaces. If ${\Phi}_2$ is taken suitably, by considering smooth homotopies preserving the $C^{\infty}$ equivalence classes, then we obtain an identical special generic map as presented in FIGURE \ref{fig:9} where ${\Phi}_{1,2}$ is a diffeomorphism on the standard sphere of the inverse image obtained by using the original two diffeomorphisms canonically. More precisely, we take diffeomorphisms ${\Phi}_1$ and ${\Phi}_2$ so that (after the deformation by a homotopy) they are extensions of identity maps on embedded standard closed discs as mentioned in Proposition \ref{prop:5} etc. and obtain a new diffeomorphism ${\Phi}_{1,2}$ on a new standard sphere by gluing the original two diffeomorphisms on the complements of the identity maps canonically, and using a suitable isotopy on the obtained standard sphere. Note also that $\Phi_2$ is needed to be chosen suitably so that the $C^{\infty}$ equivalence classes of special generic maps considered are invariant and that the resulting special generic maps through the lifts and deformations by homotopies are identical and as in FIGURE \ref{fig:9}. 
\begin{figure}
\includegraphics[width=30mm]{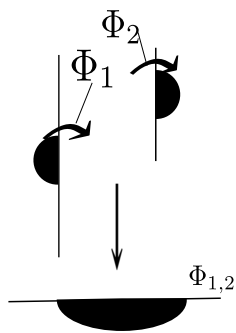}
\caption{A deformation of a special generic map obtained as a lift of a normal spherical Morse function by a homotopy preserving the $C^{\infty}$ equivalence class.}
\label{fig:9}
\end{figure}
We present this observation as a theorem.
\begin{Thm}
For any normal spherical Morse function on a closed manifold of dimension $m>6$ which is not special generic and satisfying the condition that the orientation preserving diffeomorphism group of $S^{m-1}$ is not connected, then there exists another normal spherical Morse function on the manifold such that the original function and this function is not regularly equivalent, that they are $C^{0}$ equivalent and that the obtained lifts being special generic maps into the plane as Theorem \ref{thm:3} can be $C^{\infty}$ equivalent. In addition, for an integer $n$ satisfying the relation $m \geq n \geq 2$, if for any orientation preserving diffeomorphism of $D^{m-1}$ fixing all the points on the boundary, the Gromoll filtration number is larger than $n-1$, then the obtained lifts being special generic maps into ${\mathbb{R}}^n$ as Theorem \ref{thm:3} can be $C^{\infty}$ equivalent. 
\end{Thm}

\end{document}